\long\def\symbolfootnote[#1]#2{\begingroup\def\thefootnote{\fnsymbol{footnote}}
\footnote[#1]{#2}\endgroup}

\documentclass{amsart}
\usepackage{amssymb}
\usepackage{amsfonts}
\usepackage{amsmath}

\newtheorem{theorem}{Theorem}[section]
\newtheorem{corollary}[theorem]{Corollary}
\newtheorem{lemma}[theorem]{Lemma}
\theoremstyle{remark}

\theoremstyle{definition}

\theoremstyle{proposition}
\newtheorem{proposition}[theorem]{Proposition}
\numberwithin{equation}{section}

\begin{document}
\title{K\"{a}hler manifolds with real holomorphic vector fields}
\author{Ovidiu Munteanu and Jiaping Wang}
\maketitle

\begin{abstract}
For a K\"{a}hler manifold endowed with a weighted measure $e^{-f}\,dv,$ the
associated weighted Hodge Laplacian $\Delta _{f}$ maps the space of $(p,q)$%
-forms to itself if and only if the $(1,0)$-part of the gradient vector
field $\nabla f$ is holomorphic. We use this fact to prove that for such $f$%
, a finite energy $f-$harmonic function must be pluriharmonic. Motivated by
this result, we verify that the same also holds true for $f$-harmonic maps
into a strongly negatively curved manifold. Furthermore, we demonstrate that
such $f$-harmonic maps must be constant if $f$ has an isolated minimum
point. In particular, this implies that for a compact K\"{a}hler manifold
admitting such a function, there is no nontrivial homomorphism from its
first fundamental group into that of a strongly negatively curved manifold.
\end{abstract}

\symbolfootnote[0]{The first author is partially supported by NSF grant No.
	DMS-1262140 and the second author by NSF grant No. DMS-1105799}

In this paper, $\left( M,g\right) $ denotes a K\"{a}hler manifold of complex
dimension $m$ with metric $g$ and complex structure $J.$ For a smooth real
valued function $f\in C^{\infty }\left( M\right) ,$ introduce a weighted
measure of the form $dv_{f}:=e^{-f}\,dv,$ where $dv$ is the volume form
induced from the metric $g.$ With respect to the weighted volume form $%
dv_{f},$ the adjoint $d_{f}^{\ast }$ of the exterior differential $d$ acting
on $\Omega ^{p}\left( M\right) ,$ the space of $p$ forms on $M,$ is defined
by

\begin{equation*}
\int_{M}\left\langle d\omega ,\theta \right\rangle
e^{-f}=\int_{M}\left\langle \omega ,d_{f}^{\ast }\theta \right\rangle e^{-f},
\end{equation*}%
for all $\omega \in \Omega ^{p}\left( M\right) $ and $\theta \in \Omega
^{p+1}\left( M\right) $. The weighted Hodge Laplacian $\Delta_f$ is then
given by

\begin{equation*}
\Delta _{f}:=dd_{f}^{\ast }+d_{f}^{\ast }d.
\end{equation*}%
Denote by $A^{p,q}\left( M\right) $ the space of $\left( p,q\right) $-forms
on $\left( M,g\right) .$ It is well known that the Hodge Laplacian $\Delta $
preserves the type of forms, i.e., $\Delta \omega \in A^{p,q}\left( M\right) 
$ for any $\omega \in A^{p,q}\left( M\right) .$ This fact is important in
the Hodge theory of K\"{a}hler manifolds. One may ask if the same holds true
for the weighted Hodge Laplacian $\Delta _{f}.$ Obviously, the case when
both $p$ and $q$ are zero is trivially true. When $0<p+q<2m$ we note the
following result.

\begin{proposition}
\label{pq}Let $\left( M,g\right) $ be a K\"{a}hler manifold and $f\in
C^{\infty }\left( M\right) .$ Then the weighted Hodge Laplacian $\Delta _{f}$
maps the space $A^{p,q}\left( M\right) $ of $\left( p,q\right) $-forms into
itself if and only if $\nabla f$ is real-holomorphic.
\end{proposition}

Here, $\nabla f$ is said to be real holomorphic if its $(1,0)$-part is a
holomorphic vector field. In terms of local complex coordinates $%
\{z_{1},z_{2},\cdots ,z_{m}\},$ that means that the complex vector field

\begin{equation*}
X:=g^{i\bar{k}}\frac{\partial f}{\partial \bar{z}^{k}}\frac{\partial } {%
\partial z^{i}}
\end{equation*}%
is holomorphic. This is equivalent to $f_{kl}=0$ for all $k,l\in \left\{
1,2,..,m\right\}$, in any local unitary frame $\left\{ v_{k}\right\}
_{k=1,..,m}.$ An alternative characterization is that $J\left( \nabla
f\right) $ is a Killing vector field.

There are quite a few important classes of K\"{a}hler manifolds admitting a
function with real holomorphic gradient vector field. One notable class is
the gradient K\"{a}hler-Ricci solitons. Recall that a manifold $\left(
M,g\right) $ is a gradient Ricci soliton if there exists a function $f\in
C^{\infty }\left( M\right) $ such that its Ricci curvature and the Hessian
of $f$ satisfy $\mathrm{Ric}+\mathrm{Hess}\left( f\right) =\lambda \,g,$ for
some $\lambda \in \mathbb{R}.$ Since $\left( M,g\right) $ is K\"{a}hler, the
equation can be expressed into $R_{i\bar{j}}+f_{i\bar{j}}=\lambda \,g_{i\bar{%
j}}$ and $f_{ij}=0$ under unitary frames. In particular, $\nabla f$ is a
real holomorphic vector field. Another important class arises from Calabi's
extremal K\"{a}hler metric \cite{Ca}. On a compact K\"{a}hler manifold $%
(N,g_{0}),$ one considers the following functional over the fixed K\"{a}hler
class determined by $g_{0}$

\begin{equation*}
F(g)=\int_{N}s^{2}(g)\,dv_{g},
\end{equation*}%
where $s(g)$ and $dv_{g}$ are the scalar curvature and the volume form of
metric $g,$ respectively. A critical point of this functional is called an
extremal metric. It is shown by Calabi \cite{Ca} that a metric $g$ is
extremal if and only if $\nabla s(g)$ is real holomorphic. The last class we
mention comes from eigenvalue estimates \cite{Ud, Ur}. For a compact K\"{a}%
hler manifold with Ricci curvature bounded below by a positive constant $k,$
it says that the gradient vector field of the corresponding eigenfunction
must be real holomorphic if the first nonzero eigenvalue achieves its
optimal lower bound $2k.$ In these examples, the existence of a real
holomorphic vector field is required in the study of some important
geometric questions. In \cite{G}, the existence of a function whose gradient
is real holomorphic was also needed for obtaining the strong
hypercontractivity of the weighted Laplacian. Inspired by this important
work of Gross, a complete description of possible functions with this
property on the complex hyperbolic space was obtained in \cite{GQ}.

We now briefly mention some previous results concerning manifolds admitting
real holomorphic vector fields. In an influential paper \cite{F}, Frankel
has shown that a one-parameter group of isometries acting on a K\"ahler
manifold $M$ must be Hamiltonian, i.e., induced by a Killing vector field of
the form $J(\nabla f)$ for some function $f,$ if $M$ is simply connected or
the action has nonempty fixed point set $Z$. Moreover, the Betti numbers of $%
M$ can be computed from those of the fixed point set $Z.$ Later, in \cite{CL}%
, it was shown that the Dolbeault cohomology $H^{p,q}(M)=0$ for $|p-q|>
\dim_{\mathbb{C}} Z.$ More recently, our studies (\cite{MW1, MW2}) show that
the existence of a smooth function $f$ such that $\nabla f$ is real
holomorphic has important implications on the function theory of the
manifold. In particular, it leads to various Liouville theorems for
holomorphic or, more generally, harmonic functions on $M.$ The interesting
feature is that no curvature assumption is involved.

Here, we continue our investigation of manifolds with a real holomorphic
vector field. First, we will use Proposition \ref{pq} to prove the following
Liouville theorem.

\begin{theorem}
\label{funct}Let $\left( M,g\right) $ be complete K\"{a}hler manifold and
suppose there exists $f\in C^{\infty }\left( M\right) $ so that $\nabla f$
is real holomorphic. Suppose that $u$ is an $f$-harmonic function on $M$
with finite total weighted energy $\int_{M}\left\vert \nabla u\right\vert
^{2}e^{-f}<\infty .$ Then $u$ is pluriharmonic. If, in addition, $f$ is
proper, then $u$ is constant on $M$.
\end{theorem}

Theorem \ref{funct} was first established in \cite{MW2} under a growth
assumption on $f.$ It was used there to show that shrinking gradient K\"{a}%
hler-Ricci solitons must be connected at infinity. Our approach here enables
us to remove this extra assumption.

In view of Theorem \ref{funct}, it is natural to investigate the more
general situation of harmonic maps between K\"{a}hler manifolds. We will
show that the existence of a real holomorphic vector field on $M$ implies
analogous results for harmonic maps from $M$ to another manifold $N$ with
negative curvature in a suitable sense. As is well known (see Schoen and Yau 
\cite{SY}), this leads to topological information of manifold $M.$ More
precisely, we have the following result.

\begin{theorem}
\label{maps}Let $\left( M,g\right) $ be a complete K\"{a}hler manifold with
a real holomorphic vector field $\nabla f$ for some $f\in C^{\infty }\left(
M\right) .$ Assume in addition that there exists an isolated minimum point $%
x_{0}\in M$ for $f.$ Then any $f-$harmonic map $u:M\rightarrow N$ of finite
total weighted energy into a K\"{a}hler manifold with strongly seminegative
curvature must be constant.
\end{theorem}

We recall after \cite{S} that the curvature $K_{a\bar{b}c\bar{d}}$ of a K%
\"{a}hler manifold is strongly seminegative if 
\begin{equation*}
K_{a\bar{b}c\bar{d}}\left( A^{a}\overline{B^{b}}-C^{a}\overline{D^{b}}%
\right) \left( \overline{A^{d}\overline{B^{c}}-C^{d}\overline{D^{c}}}\right)
\geq 0,
\end{equation*}%
for all complex numbers $A^{a},B^{a},C^{a},D^{a}.$ We remark that no
assumption on the curvature of $M$ or the growth of $f$ is involved in the
theorem. 

The assumption that $f$ has an isolated minimum point is indeed necessary.
To see this, consider a K\"{a}hler manifold $N$ and let $M=N\times \mathbb{C}
$. The function $f$ is taken to be constant on $N$ and $\left\vert
z\right\vert ^{2}$ on $\mathbb{C}$, so $\nabla f$ is clearly real
holomorphic. Obviously, the projection map $\pi :M\rightarrow N$ is a
nonconstant weighted harmonic map from $M$ to $N.$ 

Examples of manifolds verifying the assumptions of Theorem \ref{maps}
are abundant. They include steady K\"{a}hler Ricci solitons with positive Ricci curvature 
and scalar curvature going to zero at infinity as the potential function is strictly convex
and attains its minimum value at its only critical point (see \cite{CC}). 
Such solitons have been constructed in \cite{C}. They also include the complex projective
spaces and complex hyperbolic spaces studied in \cite{GQ}. For
example, on the unit ball model of the complex hyperbolic space $\mathbb{CH}%
^{m}$ with K\"{a}hler form $\omega =-\partial \bar{\partial}\log \left(
1-\left\vert z\right\vert ^{2}\right) ,$ the weight $f\left( z\right) =\frac{%
1}{1-\left\vert z\right\vert ^{2}}$ obviously has real holomorphic gradient and an
isolated minimum at $z=0$.

As a consequence of Theorem \ref{maps} we get the following result
concerning the fundamental group of such manifolds.

\begin{corollary}
Let $\left( M,g\right) $ be a compact K\"{a}hler manifold and assume there
exists $f$ which satisfies the assumptions in Theorem \ref{maps}. Then there
is no non-trivial homomorphism from $\pi _{1}\left( M\right) $ into that of
a compact K\"{a}hler manifold with strongly seminegative curvature.
\end{corollary}

$f-$harmonic maps have been well studied in the literature, as they are
natural objects in the presence of a smooth measure on a manifold. The
interested reader may consult \cite{HLX, RV} for some recent progress and a
more extensive reference list.

Finally, in the last part of the paper, we prove a vanishing theorem for
holomorphic forms. This result does not seem to follow from the previous work of 
\cite{CL, F, H, W} even in the compact case as we impose no assumption on
the size of the critical point set of $f.$

\begin{theorem}
Let $\left( M,g\right) $ be a complete K\"{a}hler manifold with a bounded,
real holomorphic vector field $\nabla f$ for some $f\in C^{\infty }\left(
M\right) .$ Assume in addition that there exists an isolated minimum point $%
x_{0}\in M$ for $f.$ Then, for any $p\geq 0,$ all $L^{2}$ holomorphic $(p,0)$%
-forms on $M$ must be zero.
\end{theorem}

This theorem implies that a compact K\"{a}hler manifold admitting such a
function has first Betti number equal to zero. It would be interesting to
infer some information about higher Betti numbers, under the same
assumptions.

\section{The weighted Laplacian and forms}

In this section, we prove Theorem \ref{funct}. We begin by setting up the
notations. First, to be consistent with our notation in previous works, given 
$ds^{2}:=g_{k\bar{j}}dz^{k}d\bar{z}^{j}$ a K\"{a}hler metric on $M,$ the
Riemannian metric that we consider is $4\mathrm{Re}\left( ds^{2}\right) .$
So, with respect to this Riemannian metric, we have%
\begin{equation*}
\left\vert \nabla u\right\vert ^{2}=g^{k\bar{j}}u_{k}u_{\bar{j}}\text{ \ \
and }\Delta u=g^{k\bar{j}}u_{k\bar{j}}.
\end{equation*}%
Any $\omega \in A^{p,q}\left( M\right) $ will be written locally as 
\begin{equation*}
\omega =\frac{1}{p!q!}\omega _{I\bar{J}}dz^{I}\wedge d\bar{z}^{J},
\end{equation*}%
where $\left\vert I\right\vert =p$ and $\left\vert J\right\vert =q$. On $%
A^{p,q}\left( M\right) $ we use the metric to define a Hermitian product by 
\begin{equation*}
\left\langle \omega ,\theta \right\rangle :=\frac{1}{2^{p+q}}\frac{1}{p!q!}%
g^{I\bar{K}}g^{L\bar{J}}\omega _{I\bar{J}}\overline{\theta _{K\bar{L}}}.
\end{equation*}%
The differential $d:\Omega ^{p}\left( M\right) \rightarrow \Omega
^{p+1}\left( M\right) $ acting on $p$ forms on $M$, given by 
\begin{equation*}
d\omega =dx^{k}\wedge \nabla _{\frac{\partial }{\partial x^{k}}}\omega ,
\end{equation*}%
is decomposed as $d=\partial +\bar{\partial},$ where $\partial
:A^{p,q}\left( M\right) \rightarrow A^{p+1,q}\left( M\right) $ and $\bar{%
\partial}:A^{p,q}\left( M\right) \rightarrow A^{p,q+1}\left( M\right) $ are
given by 
\begin{eqnarray*}
\partial \omega &=&dz^{k}\wedge \nabla _{\partial _{k}}\omega \\
\bar{\partial}\omega &=&d\bar{z}^{k}\wedge \nabla _{\bar{\partial}%
_{k}}\omega .
\end{eqnarray*}%
We start to denote $\partial _{k}:=\frac{\partial }{\partial z^{k}}$ and $%
\bar{\partial}_{k}:=\frac{\partial }{\partial \bar{z}^{k}}$. These operators
have adjoints $d^{\ast },\partial ^{\ast }$ and $\bar{\partial}^{\ast },$
respectively. We also have that $d^{\ast }=\partial ^{\ast }+$ $\bar{\partial%
}^{\ast }.$ We recall their well known formulas:%
\begin{eqnarray*}
d^{\ast } &=&-g^{\alpha \beta }\,i\left( \frac{\partial }{\partial x^{\alpha
}}\right) \nabla _{\frac{\partial }{\partial x^{\beta }}} \\
\partial ^{\ast } &=&-\frac{1}{2}g^{k\bar{j}}\,i\left( \partial _{k}\right)
\nabla _{\bar{\partial}_{j}} \\
\bar{\partial}^{\ast } &=&-\frac{1}{2}g^{k\bar{j}}\,i\left( \bar{\partial}%
_{j}\right) \nabla _{\partial _{k}}
\end{eqnarray*}

Here $i\left( X\right) \omega $ denotes the interior product of $\omega $ by 
$X,$ and $\alpha ,\beta \in \left\{ 1,..,2m\right\} $ are used to denote
real coordinate indices. From now on, we use normal complex coordinates at
the point under consideration. So $g_{i\bar{j}}=\delta _{i\bar{j}}$ and $%
\nabla g_{i\bar{j}}=0$ at the point. The Hodge Laplacian 
\begin{equation*}
\Delta :=dd^{\ast }+d^{\ast }d
\end{equation*}%
is positive and self adjoint. One can also define two other operators 
\begin{equation*}
\Delta ^{\partial }:=\partial \partial ^{\ast }+\partial ^{\ast }\partial 
\text{ \ and }\Delta ^{\bar{\partial}}:=\bar{\partial}\bar{\partial}^{\ast }+%
\bar{\partial}^{\ast }\bar{\partial},
\end{equation*}%
which map $A^{p,q}\left( M\right) $ into itself. The fact that $\left(
M,g\right) $ is K\"{a}hler implies 
\begin{equation*}
\Delta =\Delta ^{\partial }=\Delta ^{\bar{\partial}}.
\end{equation*}%
In particular, $\Delta $ preserves the space $A^{p,q}\left( M\right) $.

Now let us assume we have a weight $f\in C^{\infty }\left( M\right) ,$ which
gives us a new volume form $dv_{f}:=e^{-f}dv.$ We then have the
corresponding adjoint operators $d_{f}^{\ast },\partial _{f}^{\ast }$ and $%
\bar{\partial}_{f}^{\ast }.$ For example, 
\begin{equation*}
\int_{M}\left\langle d\omega ,\theta \right\rangle
e^{-f}=\int_{M}\left\langle \omega ,d_{f}^{\ast }\theta \right\rangle e^{-f}.
\end{equation*}%
The corresponding formulas for these operators are easy to find:%
\begin{eqnarray}
d_{f}^{\ast } &=&d^{\ast }+i\left( \nabla f\right)  \label{f1} \\
\partial _{f}^{\ast } &=&\partial ^{\ast }+i\left( \nabla ^{1,0}f\right) 
\notag \\
\bar{\partial}_{f}^{\ast } &=&\bar{\partial}^{\ast }+i\left( \nabla
^{0,1}f\right) ,  \notag
\end{eqnarray}%
where 
\begin{equation*}
\nabla ^{1,0}f:=\frac{1}{2}g^{j\bar{k}}f_{\bar{k}}\partial _{j}\ \ \ \text{%
and\ \ \ }\nabla ^{0,1}f:=\frac{1}{2}g^{j\bar{k}}f_{j}\partial _{\bar{k}}.
\end{equation*}%
Again, it holds that

\begin{equation*}
\nabla f=\nabla ^{1,0}f+\nabla ^{0,1}f\text{ \ \ and \ }d_{f}^{\ast
}=\partial _{f}^{\ast }+\bar{\partial}_{f}^{\ast }.
\end{equation*}%
From (\ref{f1}) it is easy to deduce the following formulas for the weighted
Hodge Laplacian $\Delta _{f}:=dd_{f}^{\ast }+d_{f}^{\ast }d,$ for $\Delta
_{f}^{\partial }:=\partial \partial _{f}^{\ast }+\partial _{f}^{\ast
}\partial $ \ and for $\Delta _{f}^{\bar{\partial}}:=\bar{\partial}\bar{%
\partial}_{f}^{\ast }+\bar{\partial}_{f}^{\ast }\bar{\partial}.$ 
\begin{eqnarray}
\Delta _{f} &=&\Delta +\mathcal{L}_{\nabla f}  \label{f2} \\
\Delta _{f}^{\partial } &=&\Delta ^{\partial }+\partial i\left( \nabla
^{1,0}f\right) +i\left( \nabla ^{1,0}f\right) \partial   \notag \\
\Delta _{f}^{\bar{\partial}} &=&\Delta ^{\bar{\partial}}+\bar{\partial}%
i\left( \nabla ^{0,1}f\right) +i\left( \nabla ^{0,1}f\right) \bar{\partial}.
\notag
\end{eqnarray}%
We have denoted by $\mathcal{L}$ the Lie derivative. Using (\ref{f2}), we
can state the necessary and sufficient condition on $f$ so that $\Delta _{f}$
maps $\left( p,q\right) $ forms to $\left( p,q\right) $ forms, cf. \cite{HM}%
. Since this is clearly true for functions, from now on we let $0<p+q<2m.$

\begin{proposition}
\label{pq_proof}The weighted Hodge Laplacian $\Delta _{f}$ preserves the
space of $\left( p,q\right) $ forms $A^{p,q}\left( M\right) $ if and only if 
$\nabla f$ is real holomorphic. In this case, 
\begin{equation*}
\Delta _{f}=\Delta _{f}^{\partial }+\Delta _{f}^{\bar{\partial}}-\Delta .
\end{equation*}
\end{proposition}

\begin{proof}
Note that 
\begin{eqnarray*}
\mathcal{L}_{\nabla f} &=&\left( di\left( \nabla ^{1,0}f\right) +i\left(
\nabla ^{1,0}f\right) d\right) +\left( di\left( \nabla ^{0,1}f\right)
+i\left( \nabla ^{0,1}f\right) d\right) \\
&=&\left( \partial i\left( \nabla ^{1,0}f\right) +i\left( \nabla
^{1,0}f\right) \partial \right) +\left( \bar{\partial}i\left( \nabla
^{0,1}f\right) +i\left( \nabla ^{0,1}f\right) \bar{\partial}\right) \\
&&+\left( \bar{\partial}i\left( \nabla ^{1,0}f\right) +i\left( \nabla
^{1,0}f\right) \bar{\partial}\right) +\left( \partial i\left( \nabla
^{0,1}f\right) +i\left( \nabla ^{0,1}f\right) \partial \right) \\
&=&\left( \Delta _{f}^{\partial }-\Delta ^{\partial }\right) +\left( \Delta
_{f}^{\bar{\partial}}-\Delta ^{\bar{\partial}}\right) +S \\
&=&\Delta _{f}^{\partial }+\Delta _{f}^{\bar{\partial}}-2\Delta +S,
\end{eqnarray*}%
where 
\begin{eqnarray*}
S &=&\left( \bar{\partial}i\left( \nabla ^{1,0}f\right) +i\left( \nabla
^{1,0}f\right) \bar{\partial}\right) +\left( \partial i\left( \nabla
^{0,1}f\right) +i\left( \nabla ^{0,1}f\right) \partial \right) \\
&=&S_{1}+S_{2}.
\end{eqnarray*}

According to (\ref{f2}), we find that 
\begin{equation*}
\Delta _{f}=\Delta _{f}^{\partial }+\Delta _{f}^{\bar{\partial}}-\Delta +S.
\end{equation*}%
Hence, we can finish the proof by computing $S\left( \omega \right) $ for $%
\omega \in A^{p,q}\left( M\right) .$ We fist note that in local coordinates 
\begin{eqnarray*}
S_{1}\left( \omega \right) &=&\bar{\partial}i\left( \nabla ^{1,0}f\right)
\omega +i\left( \nabla ^{1,0}f\right) \bar{\partial}\omega \\
&=&\frac{1}{2}f_{\bar{k}\bar{j}}d\bar{z}^{j}\wedge i\left( \partial
_{k}\right) \omega +\frac{1}{2}f_{\bar{k}}\left( \bar{\partial}i\left(
\partial _{k}\right) +i\left( \partial _{k}\right) \bar{\partial}\right)
\omega \\
&=&\frac{1}{2}f_{\bar{k}\bar{j}}d\bar{z}^{j}\wedge i\left( \partial
_{k}\right) \omega ,
\end{eqnarray*}%
where we have used the fact that $\left( \bar{\partial}i\left( \partial
_{k}\right) +i\left( \partial _{k}\right) \bar{\partial}\right) \omega =0.$
We compute in a similar fashion and obtain 
\begin{equation*}
S_{2}\left( \omega \right) =f_{kj}dz^{j}\wedge i\left( \partial _{\bar{k}%
}\right) \omega .
\end{equation*}%
Hence, $S\left( \omega \right) =0$ if and only if $S_{1}\left( \omega
\right) =S_{2}\left( \omega \right) =0$ for all $\omega \in A^{p,q}\left(
M\right) $. This happens if and only if $f_{jk}=f_{\bar{j}\bar{k}}=0$, which
is the same as $\nabla f$ being real holomorphic.
\end{proof}

We now use this result to demonstrate Theorem \ref{funct}. In fact, we will
prove a stronger statement. Let us denote $B_{x_{0}}\left( R\right) $
the geodesic ball centered at point $x_{0}$ of radius $R>0$.

\begin{theorem}
\label{funct_proof}Let $\left( M,g\right) $ be complete K\"{a}hler manifold
and $f\in C^{\infty }\left( M\right) $ with $\nabla f$ real holomorphic. 
Suppose that $u$ is an $f$-harmonic function on $M$
and that there exists a constant $C>0$ so that 
\begin{equation*}
\int_{B_{x\,_{0}}\left( R\right) }\left\vert \nabla u\right\vert
^{2}e^{-f}\leq CR^{2},
\end{equation*}%
for all $R\geq R_{0}.$ Then $u$ is pluriharmonic. If, in addition, $f$ is
proper, then $u$ is constant on $M$.
\end{theorem}

\begin{proof}
For an $f$-harmonic function $u$, it can be checked that the $1-$form $\omega :=du$ is also 
$f$-harmonic, $\Delta _{f}\omega =0$.
 However, by splitting $\omega =\partial
u+\bar{\partial}u\,\ $into $\left( 1,0\right) $ and $\left( 0,1\right) $
components, we find that $\Delta _{f}\partial u=0$ as $\Delta _{f}$
preserves the $\left( 1,0\right) $ and $\left( 0,1\right) $ forms by
Proposition \ref{pq_proof}. 

So the $\left( 1,0\right) $ form $\theta :=\partial u$ verifies $\Delta
_{f}\theta =0$ and has growth rate%
\begin{equation*}
\int_{B_{x_{0}}\left( R\right) }\left\vert \theta \right\vert ^{2}e^{-f}\leq
CR^{2}.
\end{equation*}%
Let $\phi $ be the cut-off with support in $B_{x_{0}}\left( 2R\right) $
defined by 
\begin{equation*}
\phi \left( x\right) =\left\{ 
\begin{array}{c}
1 \\ 
\frac{1}{R}\left( 2R-d\left( x_{0},x\right) \right) 
\end{array}%
\right. 
\begin{array}{l}
\text{on }B_{x_{0}}\left( R\right)  \\ 
\text{on }B_{x_{0}}\left( 2R\right) \backslash B_{x_{0}}\left( R\right) 
\end{array}%
\end{equation*}%
By $\Delta _{f}\theta =0$ we see that 
\begin{eqnarray}
0 &=&\int_{M}\left\langle \left( dd_{f}^{\ast }+d_{f}^{\ast }d\right) \theta
,\phi ^{2}\theta \right\rangle e^{-f}  \label{f3} \\
&=&\int_{M}\left\langle d_{f}^{\ast }\theta ,d_{f}^{\ast }\left( \phi
^{2}\theta \right) \right\rangle e^{-f}+\int_{M}\left\langle d\theta
,d\left( \phi ^{2}\theta \right) \right\rangle e^{-f}  \notag \\
&\geq &\int_{M}\left\vert d\theta \right\vert ^{2}\phi
^{2}e^{-f}+\int_{M}\left\vert d_{f}^{\ast }\theta \right\vert ^{2}\phi
^{2}e^{-f}-2\int_{M}\left\vert d_{f}^{\ast }\theta \right\vert \left\vert
\theta \right\vert \phi \left\vert \nabla \phi \right\vert  e^{-f} \notag \\
&&-2\int_{M}\left\vert d\theta \right\vert \left\vert \theta \right\vert
\phi \left\vert \nabla \phi \right\vert e^{-f},  \notag
\end{eqnarray}%
where in the last line we have used the Cauchy-Schwarz inequality and that 
\begin{eqnarray*}
d\left( \phi ^{2}\theta \right)  &=&\phi ^{2}d\theta +d\phi ^{2}\wedge
\theta  \\
d_{f}^{\ast }\left( \phi ^{2}\theta \right)  &=&\phi ^{2}d_{f}^{\ast }\theta
-i\left( \nabla \phi ^{2}\right) \theta .
\end{eqnarray*}%
It follows from (\ref{f3}) that 
\begin{equation}
\int_{M}\left\vert d\theta \right\vert ^{2}\phi
^{2}e^{-f}+\int_{M}\left\vert d_{f}^{\ast }\theta \right\vert ^{2}\phi
^{2}e^{-f}\leq 8\int_{M}\left\vert \theta \right\vert ^{2}\left\vert \nabla
\phi \right\vert ^{2}e^{-f}.  \label{f4}
\end{equation}%
Since by the assumption $\int_{M}\left\vert \theta \right\vert ^{2}\left\vert
\nabla \phi \right\vert ^{2}e^{-f}\leq 4C,$ (\ref{f4}) implies that 
\begin{equation}
\int_{M}\left\vert d\theta \right\vert ^{2}e^{-f}+\int_{M}\left\vert
d_{f}^{\ast }\theta \right\vert ^{2}e^{-f}<\infty .  \label{f5}
\end{equation}

Using (\ref{f3}) again, we obtain 
\begin{eqnarray*}
&&\int_{M}\left\vert d\theta \right\vert ^{2}\phi
^{2}e^{-f}+\int_{M}\left\vert d_{f}^{\ast }\theta \right\vert ^{2}\phi
^{2}e^{-f} \\
&\leq &\frac{2}{R}\int_{B_{x_{0}}\left( 2R\right) \backslash B_{x_{0}}\left(
R\right) }\left\vert d_{f}^{\ast }\theta \right\vert \left\vert \theta
\right\vert  e^{-f}+\frac{2}{R}\int_{B_{x_{0}}\left( 2R\right) \backslash
B_{x_{0}}\left( R\right) }\left\vert d\theta \right\vert \left\vert \theta
\right\vert  e^{-f}\\
&\leq &\frac{2}{R}\left( \int_{B_{x_{0}}\left( 2R\right) \backslash
B_{x_{0}}\left( R\right) }\left\vert \theta \right\vert ^{2} e^{-f}\right) ^{\frac{1%
}{2}}\left( \int_{B_{x_{0}}\left( 2R\right) \backslash B_{x_{0}}\left(
R\right) }\left\vert d_{f}^{\ast }\theta \right\vert ^{2} e^{-f}\right) ^{\frac{1}{2%
}} \\
&&+\frac{2}{R}\left( \int_{B_{x_{0}}\left( 2R\right) \backslash
B_{x_{0}}\left( R\right) }\left\vert \theta \right\vert ^{2} e^{-f}\right) ^{\frac{1%
}{2}}\left( \int_{B_{x_{0}}\left( 2R\right) \backslash B_{x_{0}}\left(
R\right) }\left\vert d\theta \right\vert ^{2} e^{-f}\right) ^{\frac{1}{2}} \\
&\leq &4\sqrt{C}\left( \int_{B_{x_{0}}\left( 2R\right) \backslash
B_{x_{0}}\left( R\right) }\left\vert d_{f}^{\ast }\theta \right\vert
^{2} e^{-f}\right) ^{\frac{1}{2}}+4\sqrt{C}\left( \int_{B_{x_{0}}\left( 2R\right)
\backslash B_{x_{0}}\left( R\right) }\left\vert d\theta \right\vert
^{2} e^{-f}\right) ^{\frac{1}{2}}.
\end{eqnarray*}%
Together with (\ref{f5}), this implies that $d\theta =d_{f}^{\ast }\theta
=0. $ Now that $u$ is pluriharmonic follows immediately from 
\begin{equation*}
\overline{\partial }\partial u=d\partial u=d\theta =0.
\end{equation*}%
The second conclusion that $u$ is constant follows as in \cite{MW1}. Indeed,
integrating by parts, we have 
\begin{eqnarray*}
\int_{\left\{ f\leq t\right\} }\left\vert \nabla u\right\vert ^{2}e^{-f}
&=&-\int_{\left\{ f\leq t\right\} }u\Delta _{f}ue^{-f}+\int_{\left\{
f=t\right\} }u\frac{\left\langle \nabla u,\nabla f\right\rangle }{\left\vert
\nabla f\right\vert }e^{-f} \\
&=&0.
\end{eqnarray*}%
The first integral above is zero because $\Delta _{f}u=0,$ while the second
term is zero because $u$ being pluriharmonic implies in particular that $%
\Delta u=0,$ hence $\left\langle \nabla u,\nabla f\right\rangle =0$.
\end{proof}

In \cite{L}, a result similar to Theorem \ref{funct_proof} was obtained for a harmonic function
with its Dirichlet energy grows no faster than
$o\left( R_i^{2}\right) $ on a sequence of geodesic balls of radius $R_i.$
Obviously, our result generalizes and strenghthenes this statement. 
The improvement to $O\left( R^{2}\right) $ also enables us to conclude the following.

\begin{proposition}
Let $\left( M,g,f\right) $ be a complete K\"{a}hler shrinking Ricci soliton
of complex dimension $m=2$. Then any bounded harmonic function on $M$ must
be constant.
\end{proposition}

\begin{proof}
Let $u$ be a bounded harmonic function. Then the $\left( 1,0\right) $ form $%
\theta =\partial u$ is harmonic. We claim that there exists $C>0$ so that
for all $R\geq R_{0},$ 
\begin{equation}
\int_{B_{x_{0}}\left( R\right) }\left\vert \theta \right\vert ^{2}\leq
CR^{2}.  \label{f7}
\end{equation}%
Indeed, this follows from a reverse Poincar\'{e} inequality and the fact
that $u$ is bounded. For a cut-off $\phi $ as in Theorem \ref{funct_proof}
we have 
\begin{eqnarray*}
\int_{M}\left\vert \nabla u\right\vert ^{2}\phi ^{2}
&=&-\int_{M}\left\langle \nabla u,\nabla \phi ^{2}\right\rangle u \\
&\leq &\frac{1}{2}\int_{M}\left\vert \nabla u\right\vert ^{2}\phi
^{2}+2\int_{M}u^2\left\vert \nabla \phi \right\vert ^{2} \\
&\leq &\frac{1}{2}\int_{M}\left\vert \nabla u\right\vert ^{2}\phi ^{2}+\frac{%
C}{R^{2}}\mathrm{Vol}\left( B_{x_{0}}\left( 2R\right) \right) \\
&\leq &\frac{1}{2}\int_{M}\left\vert \nabla u\right\vert ^{2}\phi
^{2}+CR^{2},
\end{eqnarray*}%
where in the last line we have used the fact that the volume growth of a shrinking Ricci
soliton is at most Euclidean by \cite{CZ}.

This proves (\ref{f7}). Now Theorem \ref{funct_proof}
implies that $u$ is pluriharmonic. The conclusion that $u$ is constant
follows from \cite{MW1}. Indeed, we may lift $u$ to a holomorphic function on the universal covering $\tilde{M}$ of $M$, which we continue to denote by $u$. So we have a bounded holomorphic function on a complete K\"{a}hler shrinking Ricci soliton of complex dimension $m=2$. According to \cite{MW1}, the space of holomorphic functions of any fixed polynomial growth order $d>0$ is finite dimensional. This implies that the space of bounded holomorphic functions is trivial. The proposition is proved.  
\end{proof}

\section{Harmonic maps}

In this section we prove Theorem \ref{maps}. We let $\left( M,g\right) $ be
a K\"{a}hler manifold of complex dimension $m$, admitting a function $f$ so
that $\nabla f$ is real holomorphic. Consider another K\"{a}hler manifold $%
\left( N,h\right) $ of complex dimension $n$. A map $u:M\rightarrow N$ is
called $f-$harmonic if $u$ is a critical point of the weighted energy 
\begin{equation*}
E_{f}\left( u\right) =\frac{1}{2}\int_{M}\left\vert du\right\vert ^{2}e^{-f}.
\end{equation*}%
with respect to any compactly supported variation of $u.$ We note that in
local coordinates, 
\begin{eqnarray*}
\left\vert du\right\vert ^{2} &=&2\left( \left\vert \partial u\right\vert
^{2}+\left\vert \bar{\partial}u\right\vert ^{2}\right) \\
&=&2\left( h_{a\bar{b}}g^{j\bar{k}}u_{j}^{a}\overline{u_{k}^{b}}+h_{a\bar{b}
}g^{j\bar{k}}u_{\bar{k}}^{a}\overline{u_{\bar{j}}^{b}}\right) .
\end{eqnarray*}%
The Euler-Lagrange equation implies 
\begin{equation*}
\tau _{f}\left( u\right) :=\tau \left( u\right) -i\left( \nabla f\right)
du=0,
\end{equation*}%
where $\tau \left( u\right) =\mathrm{div}\left( \nabla u\right) $ is the
usual tension field of $u$. In local coordinates, this means that 
\begin{equation*}
\Delta _{f}u^{a}+g^{j\overline{k}}\Gamma _{bc}^{a}u_{j}^{b}u_{\overline{k}
}^{c}=0,
\end{equation*}%
where 
\begin{equation*}
\Delta _{f}u^{a}=g^{j\bar{k}}\frac{\partial ^{2}u^{a}}{\partial
z^{j}\partial \bar{z}^{k}}-\frac{1}{2}g^{j\bar{k}}\left( u_{j}^{a}f_{\bar{k}
}+u_{\bar{k}}^{a}f_{j}\right) .
\end{equation*}%
Here the indices $a,b=1,2,..,n$ are used to indicate the local coordinates
on $N$ and $\Gamma _{bc}^{a}$ are the Christoffel symbols on $N$. We now
prove the following.

\begin{theorem}
\label{maps_proof}Let $\left( M,g\right) $ be a complete K\"{a}hler manifold
and suppose there exists $f\in C^{\infty }\left( M\right) $ so that $\nabla
f $ is real holomorphic. Assume in addition that $f$ achieves its minimum at
an isolated critical point $x_{0}\in M.$ Then any $f-$harmonic map $%
u:M\rightarrow N$ of finite total weighted energy into a K\"{a}hler manifold
with strongly seminegative curvature must be constant.
\end{theorem}

We divide the proof of this theorem in two parts, each of independent
interest. In the first lemma, we follow the ideas of Siu \cite{S}, with the
necessary modifications in the weighted case inspired by our work in \cite%
{MW2}, to show that $u$ must be pluriharmonic. This, in particular, implies
that $i\left( \nabla f\right) du=0$.

\begin{lemma}
\label{pluriharmonic} Let $\left( M,g\right) $ be a complete K\"{a}hler
manifold and suppose there exists $f\in C^{\infty }\left( M\right) $ so that 
$\nabla f$ is real holomorphic. Then any $f-$harmonic map $u:M\rightarrow N$
of finite total weighted energy into a K\"{a}hler manifold $N$ of strongly
seminegative curvature must be pluriharmonic. In particular, it is harmonic
and $i\left( \nabla f\right) du=0.$
\end{lemma}

\begin{proof}
By the hypothesis, $u:M\rightarrow N$ satisfies 
\begin{gather*}
\ \tau _{f}\left( u\right) =0 \\
\int_{M}\left\vert du\right\vert ^{2}e^{-f}<\infty .
\end{gather*}%
Consider a cut-off function $\phi $ with support in $B_{x_{0}}\left(
2R\right),$ $\phi =1$ on $B_{x_{0}}\left( R\right) $ and $\left\vert \nabla
\phi \right\vert \leq \frac{1}{R}$ on $M.$ In the argument that follows, we
write $du=\partial u+\bar{\partial}u,$ where $\partial u$ and $\bar{\partial}%
u\,$\ are given by 
\begin{equation*}
\partial u=\frac{\partial u^{a}}{\partial z^{j}}dz^{j}\otimes \frac{\partial 
}{\partial w^{a}}\text{ and }\bar{\partial}u=\frac{\partial u^{a}}{\partial 
\bar{z}^{j}}d\bar{z}^{j}\otimes \frac{\partial }{\partial w^{a}}
\end{equation*}%
with $\left\{ w^{a}\right\} _{a=1,..,n}$ being the local complex coordinates
on $N.$ We further denote 
\begin{equation*}
D\bar{\partial}u=u_{j\bar{k}}^{a}dz^{j}\otimes d\bar{z}^{k}\otimes \frac{%
\partial }{\partial w^{a}}\text{,}
\end{equation*}%
where 
\begin{equation*}
u_{j\bar{k}}^{a}:=\frac{\partial u_{\bar{k}}^{a}}{\partial z^{j}}+\Gamma
_{bc}^{a}u_{j}^{b}u_{\overline{k}}^{c}.
\end{equation*}

Integration by parts implies 
\begin{gather}
\int_{M}\left\vert D\bar{\partial}u\right\vert ^{2}\phi
^{2}e^{-f}=\int_{M}\left\vert u_{j\bar{k}}^{a}\right\vert ^{2}\phi
^{2}e^{-f}=\int_{M}u_{j\bar{k}}^{a}\overline{u_{j\bar{k}}^{a}}\phi ^{2}e^{-f}
\label{w0} \\
=-\int_{M}u_{j\bar{k}\bar{j}}^{a}\overline{u_{\bar{k}}^{a}}\phi
^{2}e^{-f}+\int_{M}u_{j\bar{k}}^{a}\overline{u_{\bar{k}}^{a}}f_{\bar{j}}\phi
^{2}e^{-f}-\int_{M}u_{j\bar{k}}^{a}\overline{u_{\bar{k}}^{a}}\left( \phi
^{2}\right) _{\bar{j}}e^{-f}.  \notag
\end{gather}%
Let us note that 
\begin{equation*}
u_{j\bar{k}\bar{j}}^{a}=\frac{\partial u_{j\bar{k}}^{a}}{\partial \bar{z}^{j}%
}-\Gamma _{\bar{j}\bar{k}}^{\bar{h}}u_{j\bar{h}}^{a}+\Gamma _{bc}^{a}u_{j%
\bar{k}}^{b}u_{\bar{j}}^{c}.
\end{equation*}

We now investigate each term in (\ref{w0}). First, a well known computation
in \cite{S} yields 
\begin{equation}
u_{j\bar{k}\bar{j}}^{a}=u_{j\bar{j}\bar{k}}^{a}+K_{bc\bar{d}}^{a}u_{j}^{b}u_{%
\bar{k}}^{c}\overline{u_{j}^{d}}-K_{bc\bar{d}}^{a}u_{j}^{b}u_{\bar{j}}^{c}%
\overline{u_{k}^{d}},  \label{*}
\end{equation}%
where 
\begin{equation*}
K_{bc\bar{d}}^{a}=\frac{\partial \Gamma _{bc}^{a}}{\partial \bar{w}^{d}}
\end{equation*}
is the curvature tensor on $N$. The hypothesis that the curvature of $N$ is
strongly seminegative implies that 
\begin{eqnarray*}
&&K_{b\bar{a}c\bar{d}}\left( u_{j}^{b}u_{\bar{k}}^{c}\overline{u_{j}^{d}}%
\overline{u_{\bar{k}}^{a}}-u_{j}^{b}u_{\bar{j}}^{c}\overline{u_{k}^{d}}%
\overline{u_{\bar{k}}^{a}}\right) \\
&=&\frac{1}{2}K_{b\bar{a}c\bar{d}}\left( u_{j}^{b}\overline{u_{\bar{k}}^{a}}%
-u_{k}^{b}\overline{u_{\bar{j}}^{a}}\right) \left( \overline{u_{j}^{d}%
\overline{u_{\bar{k}}^{c}}-u_{k}^{d}\overline{u_{\bar{j}}^{c}}}\right) \\
&\geq &0.
\end{eqnarray*}%
Therefore, from this computation and (\ref{w0}) we conclude the following 
\begin{eqnarray}
\int_{M}\left\vert D\bar{\partial}u\right\vert ^{2}\phi ^{2}e^{-f} &\leq
&-\int_{M}\tau ^{a}\left( u\right) _{\bar{k}}\overline{u_{\bar{k}}^{a}}\phi
^{2}e^{-f}+\int_{M}u_{j\bar{k}}^{a}\overline{u_{\bar{k}}^{a}}f_{\bar{j}}\phi
^{2}e^{-f}  \label{w1} \\
&&-\int_{M}u_{j\bar{k}}^{a}\overline{u_{\bar{k}}^{a}}\left( \phi ^{2}\right)
_{\bar{j}}e^{-f}.  \notag
\end{eqnarray}%
In a similar fashion, we get 
\begin{eqnarray}
\int_{M}\left\vert D\bar{\partial}u\right\vert ^{2}\phi ^{2}e^{-f} &\leq
&-\int_{M}\tau ^{a}\left( u\right) _{k}\overline{u_{k}^{a}}\phi
^{2}e^{-f}+\int_{M}u_{j\bar{k}}^{a}\overline{u_{j}^{a}}f_{k}\phi ^{2}e^{-f}
\label{w2} \\
&&-\int_{M}u_{j\bar{k}}^{a}\overline{u_{j}^{a}}\left( \phi ^{2}\right)
_{k}e^{-f}.  \notag
\end{eqnarray}%
Adding (\ref{w1}) and (\ref{w2}) and integrating by parts, we obtain that%
\begin{gather}
\int_{M}\left\vert D\bar{\partial}u\right\vert ^{2}\phi ^{2}e^{-f}\leq
\int_{M}\tau ^{a}\left( u\right) \tau _{f}^{a}\left( u\right) \phi ^{2}e^{-f}
\label{w3} \\
+\frac{1}{2}\int_{M}\tau ^{a}\left( u\right) \left( \overline{u_{\bar{k}}^{a}%
}\left( \phi ^{2}\right) _{\bar{k}}+\overline{u_{k}^{a}}\left( \phi
^{2}\right) _{k}\right) e^{-f}+\frac{1}{2}\int_{M}u_{j\bar{k}}^{a}\left( 
\overline{u_{\bar{k}}^{a}}f_{\bar{j}}+\overline{u_{j}^{a}}f_{k}\right) \phi
^{2}e^{-f}  \notag \\
-\frac{1}{2}\int_{M}u_{j\bar{k}}^{a}\left( \overline{u_{\bar{k}}^{a}}\left(
\phi ^{2}\right) _{\bar{j}}+\overline{u_{j}^{a}}\left( \phi ^{2}\right)
_{k}\right) e^{-f}.  \notag
\end{gather}%
Note the first term in the right side of (\ref{w3}) is zero as $u$ is $f-$%
harmonic. Furthermore, integration by parts implies 
\begin{gather*}
\int_{M}u_{j\bar{k}}^{a}\overline{u_{\bar{k}}^{a}}f_{\bar{j}}\phi
^{2}e^{-f}=-\int_{M}u_{j}^{a}\overline{u_{k\bar{k}}^{a}}f_{\bar{j}}\phi
^{2}e^{-f}+\int_{M}u_{j}^{a}\overline{u_{\bar{k}}^{a}}f_{\bar{j}}f_{\bar{k}%
}\phi ^{2}e^{-f} \\
-\int_{M}u_{j}^{a}\overline{u_{\bar{k}}^{a}}f_{\bar{j}}\left( \phi
^{2}\right) _{\bar{k}}e^{-f},
\end{gather*}%
where we have used the fact that $f_{jk}=f_{\bar{j}\bar{k}}=0.$

One obtains a similar formula for $\int_{M}u_{j\bar{k}}^{a}\overline{%
u_{j}^{a}} f_{k}\phi ^{2}e^{-f}.$ Putting together, we see that the third
term in the right side of (\ref{w3}) becomes 
\begin{gather}
\frac{1}{2}\int_{M}u_{j\bar{k}}^{a}\left( \overline{u_{\bar{k}}^{a}}f_{\bar{j%
}}+\overline{u_{j}^{a}}f_{k}\right) \phi ^{2}e^{-f}=-\frac{1}{2}%
\int_{M}\left( u_{j}^{a}f_{\bar{j}}+u_{\bar{k}}^{a}f_{k}\right) \overline{%
\tau ^{a}\left( u\right) }\phi ^{2}e^{-f}  \label{w4} \\
+\int_{M}\mathrm{Re}\left( u_{j}^{a}\overline{u_{\bar{k}}^{a}}f_{\bar{j}}f_{%
\bar{k}}\right) \phi ^{2}e^{-f}-\frac{1}{2}\int_{M}u_{j}^{a}\overline{u_{%
\bar{k}}^{a}}f_{\bar{j}}\left( \phi ^{2}\right) _{\bar{k}}e^{-f}-\frac{1}{2}%
\int_{M}u_{\bar{k}}^{a}\overline{u_{j}^{a}}f_{k}\left( \phi ^{2}\right) _{j} e^{-f}.
\notag
\end{gather}
Plugging into (\ref{w3}), we conclude 
\begin{gather}
\int_{M}\left\vert D\bar{\partial}u\right\vert ^{2}\phi ^{2}e^{-f}\leq -%
\frac{1}{4}\int_{M}\left\vert u_{j}^{a}f_{\bar{j}}+u_{\bar{k}%
}^{a}f_{k}\right\vert ^{2}\phi ^{2}e^{-f}+\int_{M}\mathrm{Re}\left( u_{j}^{a}%
\overline{u_{\bar{k}}^{a}}f_{\bar{j}}f_{\bar{k}}\right) \phi^2\,e^{-f} \label{w5} \\
+\int_{M}\tau ^{a}\left( u\right) \overline{i\left( \nabla \phi ^{2}\right)
du^{a}}e^{-f}-\frac{1}{2}\int_{M}u_{j}^{a}\overline{u_{\bar{k}}^{a}}f_{\bar{j%
}}\left( \phi ^{2}\right) _{\bar{k}}e^{-f}-\frac{1}{2}\int_{M}u_{\bar{k}}^{a}%
\overline{u_{j}^{a}}f_{k}\left( \phi ^{2}\right) _{j} e^{-f} \notag \\
-\frac{1}{2}\int_{M}u_{j\bar{k}}^{a}\left( \overline{u_{\bar{k}}^{a}}\left(
\phi ^{2}\right) _{\bar{j}}+\overline{u_{j}^{a}}\left( \phi ^{2}\right)
_{k}\right) e^{-f}.  \notag
\end{gather}

Note that 
\begin{equation*}
\frac{1}{4}\left\vert u_{j}^{a}f_{\bar{j}}+u_{\bar{k}}^{a}f_{k}\right\vert
^{2}-\mathrm{Re}\left( u_{j}^{a}\overline{u_{\bar{k}}^{a}}f_{\bar{j}}f_{\bar{%
k}}\right) =\frac{1}{4}\left\vert u_{j}^{a}f_{\bar{j}}-u_{\bar{k}%
}^{a}f_{k}\right\vert ^{2}\geq 0.
\end{equation*}%
Also, the last term in (\ref{w5}) can be estimated as 
\begin{equation*}
\int_{M}\left\vert u_{j\bar{k}}^{a}\overline{u_{\bar{k}}^{a}}\left( \phi
^{2}\right) _{\bar{j}}\right\vert e^{-f}\leq 2\int_{M}\left\vert D\bar{%
\partial}u\right\vert \left\vert du\right\vert \phi \left\vert \nabla \phi
\right\vert  e^{-f}.
\end{equation*}

Hence, by the Cauchy-Schwarz inequality, (\ref{w5}) becomes 
\begin{gather}
\frac{1}{2}\int_{M}\left\vert D\bar{\partial}u\right\vert ^{2}\phi
^{2}e^{-f}\leq -\frac{1}{4}\int_{M}\left\vert u_{j}^{a}f_{\bar{j}}-u_{\bar{k}%
}^{a}f_{k}\right\vert ^{2}\phi ^{2}e^{-f}+2\int_{M}\left\vert du\right\vert
^{2}\left\vert \nabla \phi \right\vert ^{2}e^{-f}  \label{w6} \\
+\int_{M}\tau ^{a}\left( u\right) \overline{i\left( \nabla \phi ^{2}\right)
du^{a}}e^{-f}-\frac{1}{2}\int_{M}u_{j}^{a}\overline{u_{\bar{k}}^{a}}f_{\bar{j%
}}\left( \phi ^{2}\right) _{\bar{k}}e^{-f}-\frac{1}{2}\int_{M}u_{\bar{k}}^{a}%
\overline{u_{j}^{a}}f_{k}\left( \phi ^{2}\right) _{j} e^{-f}.  \notag
\end{gather}%
We now deal with the other terms as follows. We have that 
\begin{gather}
\int_{M}\tau ^{a}\left( u\right) \overline{i\left( \nabla \phi ^{2}\right)
du^{a}}e^{-f}-\frac{1}{2}\int_{M}u_{j}^{a}\overline{u_{\bar{k}}^{a}}f_{\bar{j%
}}\left( \phi ^{2}\right) _{\bar{k}}e^{-f}  \label{w7} \\
-\frac{1}{2}\int_{M}u_{\bar{k}}^{a}\overline{u_{j}^{a}}f_{k}\left( \phi
^{2}\right) _{j}e^{-f}=\frac{1}{4}\int_{M}\left( u_{j}^{a}f_{\bar{j}}-u_{%
\bar{k}}^{a}f_{k}\right) \left( \overline{u_{h}^{a}\phi _{\bar{h}}^{2}-u_{%
\bar{l}}^{a}\phi _{l}^{2}}\right) e^{-f}  \notag \\
\leq \frac{1}{4}\int_{M}\left\vert u_{j}^{a}f_{\bar{j}}-u_{\bar{k}%
}^{a}f_{k}\right\vert ^{2}\phi ^{2}e^{-f}+\int_{M}\left\vert du\right\vert
^{2}\left\vert \nabla \phi \right\vert ^{2}e^{-f}.  \notag
\end{gather}%
Putting (\ref{w7}) into (\ref{w6}), we get 
\begin{equation}
\frac{1}{2}\int_{M}\left\vert D\bar{\partial}u\right\vert ^{2}\phi
^{2}e^{-f}\leq 4\int_{M}\left\vert du\right\vert ^{2}\left\vert \nabla \phi
\right\vert ^{2}e^{-f}.  \label{w8}
\end{equation}

Since $\int_{M}\left\vert du\right\vert ^{2}e^{-f}<\infty ,$ it is easy to
see that as $R\rightarrow \infty ,$%
\begin{equation*}
\int_{M}\left\vert du\right\vert ^{2}\left\vert \nabla \phi \right\vert
^{2}e^{-f}\rightarrow 0.
\end{equation*}%
Therefore, by letting $R\rightarrow \infty $ in (\ref{w8}), we conclude that 
$D\bar{\partial}u=0$ or $u$ is pluriharmonic. In particular, this implies
that $u$ is a harmonic map. Hence, $\tau (u)=0$ and $i\left( \nabla f\right)
du=0.$ This proves the lemma.
\end{proof}

We point out that Lemma \ref{pluriharmonic} also holds under energy
growth assumption on $u$ that 
\begin{equation*}
\int_{B_{x_{0}}\left( R\right) }\left\vert du\right\vert ^{2}e^{-f}\leq
CR^{2},\text{ \ for all }R\geq R_{0}\text{. }
\end{equation*}%
The argument for this improvement is similar to that of Theorem \ref%
{funct_proof}.

We next present a local result that holds for harmonic maps between any two
Riemannian manifolds. We show that $u$ must be constant if $u$ is harmonic
and $i\left( \nabla f\right) du=0.$

\begin{lemma}
\label{local}Let $\left( M,g\right) $ be a complete Riemannian manifold and $%
u:\Omega \rightarrow N$ a harmonic map from a domain $\Omega\subset M$ into
a Riemannian manifold $N.$ If $i\left( \nabla f\right) du=0$ on $\Omega$ for
a smooth function $f$ and f has unique minimum point $x_{0}\in \Omega,$ then 
$u $ must be constant.
\end{lemma}

\begin{proof}
Since this lemma is stated in a Riemannian setting, we will denote here 
\begin{equation*}
du=u_{k}^{a}dx^{k}\otimes \frac{\partial }{\partial y^{a}},
\end{equation*}%
where $\left\{ x^{k}\right\} _{k=1,..,2m}$ and $\left\{ y^{a}\right\}
_{a=1,..,2n}$ are real coordinates on $M$ and $N$. The fact that $u$ is
harmonic and $i\left( \nabla f\right) du=0$ means that%
\begin{eqnarray}
\Delta u^{a}+g^{kj}\Gamma _{bc}^{a}u_{k}^{b}u_{j}^{c} &=&0  \label{l1} \\
g^{kj}u_{k}^{a}f_{j} &=&0.  \notag
\end{eqnarray}

Let $\delta >0$ be sufficiently small so that $u\left( B_{x_{0}}\left(
\delta \right) \right) \subset B_{y_{0}}\left( \rho \right) ,$ where $%
y_{0}=u\left( x_{0}\right) ,$ and the exponential map $\exp
_{y_{0}}:B_0(\rho)\subset T_{y_{0}}N\rightarrow B_{y_{0}}\left( \rho \right) 
$ is a diffeomorphism. Under the induced normal coordinates, we have that
for $y\in B_{y_0}(\eta),$ 
\begin{equation}
|h_{ab}(y)-\delta _{ab}|\le C\,\eta \text{ and }\left\vert\frac{\partial
h_{ab}}{\partial y^{c}}\right\vert(y)\leq C\,\eta  \label{l2}
\end{equation}%
for all $\eta\leq \rho,$ where $C$ is a constant independent of $\eta.$

We normalize $f$ so that $f\left( x_{0}\right) =0.$ Since $x_{0}$ is an
isolated critical point, there exists $\varepsilon >0$ small enough so that
the level set $\left\{ f=\varepsilon \right\} $ has a connected component
completely contained in $B_{x_{0}}\left( \delta \right) .$ Denote by 
\begin{equation*}
D\left( \varepsilon \right) :=\left\{ f\leq \varepsilon \right\} \cap
B_{x_{0}}\left( \delta \right)
\end{equation*}%
and note that $\partial D\left( \varepsilon \right) $ has unit normal vector 
$\nu :=\frac{\nabla f}{\left\vert \nabla f\right\vert }$. Clearly, $%
u\left(D(\varepsilon) \right)\subset B_{y_0}(\eta)$ with $\eta\to 0$ as $%
\varepsilon\to 0.$

Integrating by parts, 
\begin{eqnarray}
\int_{D\left( \varepsilon \right) }g^{kj}\,h_{ab}\,u_{k}^{a}\,u_{j}^{b}
&=&-\int_{D\left( \varepsilon \right) }h_{ab}\,\left( \Delta u^{a}\right)
u^{b}-\int_{D\left( \varepsilon \right) }\,\left\langle
dh_{ab},du^{a}\right\rangle u^{b}  \label{l3} \\
&&+\int_{\partial D\left( \varepsilon \right) }\frac{1}{\left\vert \nabla
f\right\vert }g^{kj}\,h_{ab}\,u_{k}^{a}u^{b}f_{j}  \notag \\
&\leq &C\,\eta \int_{D\left( \varepsilon \right) }\left\vert du\right\vert
^{2},  \notag
\end{eqnarray}%
where we have used (\ref{l1}) so that the boundary term is zero, and that 
\begin{equation*}
|\Delta u^{a}|=|g^{kj}\,\Gamma _{bc}^{a}u_{k}^{b}u_{j}^{c}|\leq C\,\eta
\,|du|^{2},
\end{equation*}%
as well as 
\begin{eqnarray*}
\left\vert \left\langle dh_{ab},du^{a}\right\rangle u^{b}\right\vert
&=&\left\vert g^{kj}\frac{\partial h_{ab}}{\partial x^{k}}%
u_{j}^{a}u^{b}\right\vert =\left\vert g^{kj}\frac{\partial h_{ab}}{\partial
y^{c}}u_{j}^{a}u_{k}^{c}u^{b}\right\vert \\
&\leq &C\eta \left\vert du\right\vert ^{2}.
\end{eqnarray*}

By choosing $\varepsilon >0$ to be sufficiently small, this implies that $%
\left\vert du\right\vert =0$ on $D\left( \varepsilon \right) $. Therefore, $%
u $ is constant on $D\left( \varepsilon \right).$ By the unique continuation
property, $u$ must be constant on $\Omega.$
\end{proof}

We can now prove Theorem \ref{maps_proof}. Using Lemma \ref{pluriharmonic},
we see that $u:M\rightarrow N$ must be harmonic and $i\left( \nabla f\right)
du=0$ on $M$. Now Lemma \ref{local} says that $u$ must be constant on $M.$
This proves Theorem \ref{maps_proof}.

It turns out under the hypothesis in Theorem \ref{maps_proof}, we also have
Liouville property for harmonic maps, not just for weighted ones. The idea
is to show that $i\left(\nabla f\right) du=0$ again and then appeal to Lemma %
\ref{local}.

\begin{theorem}
Let $\left( M,g\right) $ be a compact K\"{a}hler manifold and suppose there
exists $f\in C^{\infty }\left( M\right) $ so that $\nabla f$ is real
holomorphic. Assume in addition that $f$ achieves its minimum at an isolated
critical point. Then any harmonic map $u:M\rightarrow N$ into a K\"{a}hler
manifold of strongly seminegative curvature must be constant.
\end{theorem}

\begin{proof}
Since $u:M\rightarrow N$ is harmonic and $N$ has strongly seminegative
curvature, Siu's theorem in \cite{S} implies that $u$ is pluriharmonic, or 
\begin{equation}
u_{j\bar{k}}^{a}=0.  \label{h1}
\end{equation}

We now define the flow induced by the vector field $J(\nabla f).$ 
\begin{eqnarray*}
\frac{d\phi _{t}}{dt} &=&J\left( \nabla f\right) \left( \phi _{t}\right) \\
\phi _{0} &=&\mathrm{Id.}
\end{eqnarray*}%
Since $\nabla f$ is real holomorphic, $J\left( \nabla f\right) $ is a
Killing vector field. So $\phi _{t}$ is a one parameter group of isometries
of $M$. In particular, 
\begin{equation*}
u_{t}:=u\circ \phi _{t}
\end{equation*}%
is a continuous family of harmonic maps from $M$ to $N$. Since $N$ has
strongly seminegative curvature, it has nonpositive sectional curvature as
well. We now use the uniqueness theorem for harmonic maps in \cite{SY1} to
show that $u_{t}=u$ for all $t\geq 0.$ Indeed, lifting $u_{t}$ to the
universal coverings $\widetilde{M}$ and $\widetilde{N},$ we get a family of
harmonic maps $\widetilde{u_{t}}:\widetilde{M}\rightarrow \widetilde{N}.$
Using the fact that $u_t$ is homotopic to $u$ and $N$ has nonpositive
curvature, a standard computation shows that $\widetilde{r}^{2}\left( 
\widetilde{u_{0}},\widetilde{u_{t}}\right) $ descends to $M$ and is
subharmonic, where $\widetilde{r}$ is the distance function on $\widetilde{N}%
.$ Therefore, for each fixed $t,$ $\widetilde{r}^{2}\left(\widetilde{u_{0}},%
\widetilde{u_{t}}\right) $ is a constant function on $M$ as $M$ is compact.
However, at the minimum point $x_{0} $ of $f,$ $\phi _{t}(x_0)=x_0.$ This
means that $u_{t}\left( x_{0}\right) =u_{0}\left( x_{0}\right)$ for all $%
t\geq 0.$ In turn, it shows that $\widetilde{r}^{2}\left( \widetilde{u_{0}}%
\left( \widetilde{x_{0}}\right) ,\widetilde{u_{t}}\left( \widetilde{x_{0}}%
\right) \right) =0$. Hence, $\widetilde{r}^{2}\left( \widetilde{u_{0}},%
\widetilde{u_{t}}\right) =0$ and $u_{0}=u_{t}$ for all $t\geq 0$.

We now differentiate the equation $u_{0}=u_{t}$ in $t$ and get that 
\begin{eqnarray*}
0 &=&\frac{d}{dt}u_{t} \\
&=&i\left( J\left( \nabla f\right) \right) du.
\end{eqnarray*}%
This means, in complex coordinates, that 
\begin{equation}
u_{k}^{a}f_{\bar{k}}=u_{\bar{k}}^{a}f_{k}.  \label{h2}
\end{equation}%
Using (\ref{h1}), we see that 
\begin{eqnarray}
\left( u_{k}^{a}f_{\bar{k}}\right) _{\bar{j}} &=&u_{k\bar{j}}^{a}f_{\bar{k}%
}+u_{k}^{a}f_{\bar{k}\bar{j}}=0  \label{h3} \\
\left( u_{\bar{k}}^{a}f_{k}\right) _{j} &=&u_{j\bar{k}}^{a}f_{k}+u_{\bar{k}%
}^{a}f_{kj}=0,  \notag
\end{eqnarray}%
where we have also made use of $f_{kj}=f_{\bar{k}\bar{j}}=0$ as $\nabla f$
is real holomorphic. By (\ref{h2}) and (\ref{h3}), we conclude that 
\begin{equation*}
\left( u_{k}^{a}f_{\bar{k}}\right) _{\bar{j}}=\left( u_{k}^{a}f_{\bar{k}%
}\right) _{j}=0.
\end{equation*}%
This forces the function $u_{k}^{a}f_{\bar{k}}$ to be constant on $M$. Since 
$f$ has a critical point on $M,$ $u_{k}^{a}f_{\bar{k}}=u_{\bar{k}%
}^{a}f_{k}=0 $ on $M$. This proves that $i\left( \nabla f\right) du=0.$ By
Lemma \ref{local}, $u$ must be constant.
\end{proof}

As mentioned earlier, the existence of a function $f$ with $\nabla f$ being
real holomorphic on $M$ has strong implications on the topology of $M$. An
early result in this direction was proved by Frankel \cite{F}, which states
that all odd Betti numbers of a compact K\"{a}hler manifold must be zero if
it has a Killing vector field whose zero set is non-empty and discrete.
Howard \cite{H} proved a result of similar nature. In particular, it says
that a projective manifold $M$ admitting a holomorphic vector field with
nonempty and discrete zero set has only trivial holomorphic $(p,0)$-forms.

We establish below a vanishing result for holomorphic forms under an
assumption that is more in the spirit of the preceding Liouville type
results.

\begin{theorem}
Let $\left( M,g\right) $ be a complete K\"{a}hler manifold. Suppose there
exists $f\in C^{\infty }\left( M\right) $ such that $\nabla f$ is real
holomorphic and bounded on $M$. Assume in addition that $f$ has an isolated
minimum point in $M.$ Then any $L^{2}$ holomorphic $(p,0)-$form on $M$ must
be zero for all $p\geq 0.$
\end{theorem}

\begin{proof}
We proceed by induction on $p.$ For $p=0,$ this is certainly true as any $%
L^{2}$ holomorphic function must be constant. Let us assume that the result
holds for $\left( p-1\right).$ We now prove it for $p-$forms. Consider 
\begin{equation*}
\omega =\frac{1}{p!}\omega _{i_{1}....i_{p}}dz^{i_{1}}\wedge ...\wedge
dz^{i_{p}},
\end{equation*}%
an $L^{2}$ holomorphic $p-$form. Now the $\left( p-1\right) $ form 
\begin{eqnarray*}
\theta &=&\omega \left( \cdot ,...,\cdot ,\nabla f\right) \\
&=&\frac{1}{\left( p-1\right) !}\left( \omega _{i_{1}....i_{p}}f_{\overline{%
i_{p}}}\right) dz^{i_{1}}\wedge ...\wedge dz^{i_{p-1}}
\end{eqnarray*}%
is holomorphic as $\nabla f$ is real holomorphic and $\omega $ is
holomorphic. It is also in $L^{2}$ as $\nabla f$ is bounded on $M.$ By the
induction hypothesis, $\theta =0.$ Hence, 
\begin{equation}
\omega _{i_{1}....i_{p}}f_{\overline{i_{p}}}=0.  \label{ho1}
\end{equation}%
The rest of the argument is local and around an isolated minimum point $%
x_{0} $ of $f.$ Note that since $\omega $ is holomorphic and in $L^{2},$ it
is also harmonic, closed and co-closed. Thus, in a fixed complex local
coordinate chart $U$ at $x_{0},$ we know that $\omega $ is exact, i.e., $%
\omega =\partial \eta $ for a $\left( p-1,0\right) $ form $\eta $ defined on 
$U$.

We normalize $f$ so that $f\left( x_{0}\right) =0.$ Since $x_{0}$ is an
isolated critical point, there exists $\varepsilon >0$ small enough so that
the level set $\left\{ f=\varepsilon \right\} $ has a connected component
completely contained in $U.$ Let 
\begin{equation*}
D\left( \varepsilon \right) :=\left\{ f\leq \varepsilon \right\} \cap U
\end{equation*}%
and note that $\partial D\left( \varepsilon \right) $ has normal vector $\nu
:=\frac{\nabla f}{\left\vert \nabla f\right\vert }$. On $D\left( \varepsilon
\right) $ we have that $\omega =\partial \eta .$ Therefore, 
\begin{eqnarray*}
\int_{D\left( \varepsilon \right) }\left\vert \omega \right\vert ^{2}
&=&\int_{D\left( \varepsilon \right) }\left\langle \omega ,\partial \eta
\right\rangle \\
&=&-\int_{D\left( \varepsilon \right) }\left\langle \partial ^{\ast }\omega
,\eta \right\rangle +\int_{\partial D\left( \varepsilon \right)
}\left\langle \omega ,df\wedge \eta \right\rangle \frac{1}{\left\vert \nabla
f\right\vert } \\
&=&0
\end{eqnarray*}%
as $\omega $ is co-closed and $\omega \left( \cdot ,...,\cdot ,\nabla
f\right) =0$. This proves that $\omega =0$ on $D\left( \varepsilon \right).$
Thus, $\omega =0$ on $M$ by the unique continuation property. This proves
the theorem.
\end{proof}

\address {\noindent Department of Mathematics\\ University of Connecticut\\
Storrs, CT 06268\\ USA\\ \email{\textit{E-mail address}: {\tt
ovidiu.munteanu@uconn.edu} } \vskip 0.3in \address {\noindent School of
Mathematics \\ University of Minnesota\\ Minneapolis, MN 55455\\ USA\\ 
\email{\textit{E-mail address}: {\tt jiaping@math.umn.edu}}


\begin{thebibliography}{99}
\bibitem{Ca} E. Calabi, Extremal K\"{a}hler metrics. Seminar on Diff. Geom.,
Ann. of Math. Stud., 102, Princeton Univ. Press, 1982.

\bibitem{C} H.D. Cao, Existence of gradient K\"{a}hler-Ricci solitons, Elliptic and parabolic methods in geometry, AK Peters, Wellesley (1996), 1-16.

\bibitem{CC} H.D. Cao and Q. Chen, On locally conformally flat gradient
steady Ricci solitons, Trans. Amer. Math. Soc. 364 (2012), 2377-2391.

\bibitem{CZ} H.D. Cao and D. Zhou, On complete gradient shrinking Ricci
solitons, J. Differential Geom. 85 (2010), no. 2, 175-186.

\bibitem{CL} J. B. Carrell and D. Lieberman, Holomorphic vector fields and K%
\"{a}hler manifolds, Invent Math. 21 (1973), 275-286.

\bibitem{F} T. Frankel, Fixed points and torsion on K\"{a}hler manifolds,
Ann. of Math 70 (1959), 1-8.

\bibitem{G} L. Gross, Hypercontractivity over complex manifolds. Acta Math.
182 (1999) 159-206.

\bibitem{GQ} L. Gross and Z. Qian, Holomorphic Dirichlet forms on complex
manifolds, Math. Z. 246 (2004) 521-561.

\bibitem{HM} S. Hall and T. Murphy, On the linear stability of K\"{a}%
hler-Ricci solitons, Proc Amer. Math. Soc., 139 No. 9 (2011), 3327-3337.

\bibitem{H} A. Howard, Holomorphic vector fields on algebraic manifolds,
Amer. J. Math. 94 (1972), 1282-1290.

\bibitem{HLX} B. Hua, S. Liu, C. Xia, Liouville theorems for f-harmonic maps
into Hadamard spaces, arXiv:1305.0485.

\bibitem{L} P. Li, On the structure of complete K\"ahler manifolds with nonnegative 
curvature near infinity, Invent. Math. 99 (1990), 579-600.


\bibitem{MW1} O. Munteanu and J.\ Wang, Holomorphic functions on K\"{a}hler
Ricci solitons, J. London Math. Soc. 89 (2014), no. 3, 817-831.

\bibitem{MW2} O. Munteanu and J.\ Wang, Topology of K\"{a}hler Ricci
solitons, to appear in J. Differential Geom.



\bibitem{RV} M. Rimoldi and G. Veronelli, Topology of steady and expanding
gradient Ricci solitons via f-harmonic maps, Diff. Geom. Appl. 31 (2013),
623-638.

\bibitem{SY} R. Schoen and S.T. Yau, Harmonic maps and the topology of
stable hypersurfaces and manifolds with non-negative Ricci curvature,
Comment. Math. Helv. 51 (1976), no. 3, 333-341.

\bibitem{SY1} R. Schoen and S.T. Yau, Compact group actions and the topology
of manifolds with non-positive curvature, Topology 18 (1979), 361-380.

\bibitem{S} Y.T. Siu, The complex-analyticity of harmonic maps and the
strong rigidity of compact K\"{a}hler manifolds, Annals of Math. 112 (1980),
no. 1, 73-111.

\bibitem{Ud} S. Udagawa, Compact Kaehler manifolds and the eigenvalues of
the Laplacian, Colloquium Math., 56, (1988), no. 2, 341-349.

\bibitem{Ur} H. Urakawa, Stability of harmonic maps and eigenvalues of the
Laplacian, Trans. Amer. Math. Soc. 301 (1987), no. 2, 557--589.

\bibitem{W} E. Wright, Killing vector fields and harmonic forms, Tran. Amer.
Math. Soc. 199 (1974), 199-202.
\end{thebibliography}
\end{document}